\documentclass[journal]{article}

\usepackage{graphicx,url}
\usepackage{dcolumn}
\usepackage{bm}
\usepackage{amsmath,amsfonts,amssymb,citesort}

\newtheorem{definition}{\noindent{\it Definition}}[section]
\newtheorem{theorem}{\noindent{\it Theorem}}[section]

\newtheorem{corollary}[theorem]{\noindent{\it Corollary}}
\newenvironment{proof}{\noindent{\it Proof:}}{$\hfill$ $\Box$\\ }
\newtheorem{example}{\noindent{\it Example}}[section]

\begin{document}

\title{On Matroids and Linearly Independent Set Families}

\author{Giuliano G. La Guardia, Luciane Grossi, Welington Santos
\thanks{The authors are with Department of Mathematics and Statistics,
State University of Ponta Grossa, 84030-900, Ponta Grossa - PR,
Brazil. Corresponding author: Giuliano G. La Guardia ({\tt \small
gguardia@uepg.br}). }}

\maketitle

\begin{abstract}
New families of matroids are constructed in this note. These new
families are derived from the concept of linearly independent set
family (LISF) introduced by Eicker and Ewald [Linear Algebra and
its Applications 388 (2004) 173–-191]. The proposed construction
generalizes in a natural way the well known class of vectorial
matroids over a field.

\end{abstract}

\section{Introduction}

In his seminal paper \cite{Whitney:1935} on matroid theory,
Hassler Whitney dealt with the problem of characterizing matroids
that are representable over a given field (see also the
interesting papers
\cite{Brualdi:1969,Brylawski:1973,Greene:1976}). In fact, as it is
well known, the matroid theory is a powerful tool in order to
study several classes endowed with algebraic structures such as,
affine spaces, vector spaces, algebraic independence, graph theory
and so on. Among these classes, a particular class is of essential
importance: the class of vectorial matroids.

In this note we generalize the class of vectorial matroid by
applying the concept of linearly independent set family (LISF),
introduced by Eicker and Ewald \cite{Eicker:2004}, which extends
in a straightforward way the definition of linearly independent
vectors to independent sets in a vector space. More precisely, the
LISF's have essential ingredients in order to provide a natural
generalization of the class of vectorial matroids over a given
field.

Section~2 presents basic concepts on matroid theory and linearly
independent set family, necessary for the development of this
note. In Section 3, we present the contributions of this paper: a
new class of matroids derived from linearly independent set
families are constructed. In Section~4, the final remarks are
drawn.

\section{Preliminaries}\label{sec2}
This section is concerned with a review of matroid theory
\cite{Welsh:1976,Oxley:1992} as well as the review of the concept
of linearly independent set family (LISF) \cite{Eicker:2004}.

\subsection{Matroid Theory}\label{sub2.1}

As was said previously we utilize the definition of matroid based
on independent sets (although the other definitions are
equivalents). The following basic concepts can be found in
\cite{Oxley:1992}.

\begin{definition}\label{matro}
A matroid $M$ is an ordered pair $(S, \mathcal{I})$ consisting of
a finite set $S$ and a collection $\mathcal{I}$ of subsets of $S$
satisfying the following three conditions:\\
(I.1) $\emptyset \in \mathcal{I}$;\\
(I.2) If $I \in \mathcal{I}$ and ${I}' \subset I$, then ${I}' \in
\mathcal{I}$;\\
(I.3) If ${I}_1$, ${I}_2$ $\in \mathcal{I}$ and $\mid {I}_1\mid <
\mid {I}_2\mid$, then there exists an element $e \in I_{2} - I_1$
such that $I_1 \cup \{e\} \in \mathcal{I}$, where $\mid .\mid$
denote the cardinality of the set.
\end{definition}

If $M$ is the matroid $(S, \mathcal{I})$, then $M$ is called
\emph{matroid on} $S$. The members of $\mathcal{I}$ are
\emph{independent sets of} $M$, and $S$ is the \emph{ground set
of} $M$. A subset of $S$ that is not in $\mathcal{I}$ is called
\emph{dependent}. Minimal dependent sets are dependents sets all
of whose proper subsets are independents. Minimal dependent sets
are called \emph{circuits} of $M$. An independent set is called
\emph{maximal} if the inclusion of any element in this set results
in a dependent set. Maximal independent sets are called
\emph{basis} of the matroid. It is well known that a matroid can
be defined in many different (but equivalent) ways, i. e., by
means of independent sets, circuits, basis and so on. In our case
we consider the definition of matroid based on independent sets,
as given above.

Let us recall the well known concept of vectorial matroid:

\begin{theorem}\label{connec}
Let $S$ be the set of column labels of a matrix $A_{m \times n}$
over a field ${\mathbb F}$, and let $\mathcal{I}$ be the set of
subsets $X$ of $S$ for which the multiset of columns labelled by
$X$ is linearly independent (LI) in $V(m, F)$, the $m$-dimensional
vector space over ${\mathbb F}$. Then $(S, \mathcal{I})$ is a
matroid.
\end{theorem}

\subsection{Linearly Independent set Families}\label{sub2.2}

The concept of linearly independent set families was introduced by
Eicker and Ewald in \cite{Eicker:2004}. This definition extends in
a straightforward way the definition of linearly independent
vectors to independent sets in a vector space. Although our
definition is different from the original one, it contains
essentially the same idea contained in \cite{Eicker:2004}.

\begin{definition}
Let ${\mathbb V}$ be a $l$-dimensional vector space over a field
${\mathbb F}$. A family $\mathcal{J}$ of non-empty subsets
$C_{i}\subset {\mathbb V}$, where $i=1,\ldots, n$ ( $n \leq l$),
given by $\mathcal{J}:= \{ C_{i}, i=1,\ldots, n \}$ is called a
\emph{linearly independent set family} (LISF) if and only if any
selection of $n$ vectors $v_{i} \in C_{i}$ is linearly independent
in ${\mathbb V}$.
\end{definition}


\begin{example}
The first and trivial example of LISF is a set of $n\leq l$
linearly independent vectors in ${\mathbb{R}}^{l}$. As a second
illustrative example, consider in ${\mathbb{R}}^{2}$ the open
quadrants $C_{1}:= \{ (x_1, x_{2}): x_1 > 0, x_{2}
> 0 \}$, and $C_{2}:= \{ (x_1,x_2): x_1 > 0, x_2 < 0 \}$. Then
$\mathcal{F}:= \{ C_{1}, C_{2} \}$ is a LISF.
\end{example}

\section{The Results}

In this section we present the contributions of the paper.
Theorems~\ref{main} and \ref{main-multdim} generalize the well
known class of vectorial matroids over a given field,
consequently, new families of matroids  are obtained.

\begin{theorem}\label{main}
Consider that $n\geq 1$ and $ l\geq 1$ are integers. Let $E_1,
E_{2}, \ldots , E_{n}$ be subsets of a finite dimensional vector
space ${\mathbb V}$ over a field ${\mathbb F}$ such that $E_{i}
\subset {\mathbb{W}}_{i}$ for all $i=1, \ldots , n$, where
${\mathbb{W}}_{i}$ are one-dimensional subspaces of ${\mathbb V}$.
Consider the multiset of labels $S =\{ 1, \ldots , n\}$, and let
$\mathcal{I}$ be the set of subsets $I = \{ i_1, \ldots , i_{j}\}$
of $S$ for which $\{E_{i_{1}}, E_{i_{2}}, \ldots , E_{i_{j}}\}$
form a LISF. Then the ordered pair $(S, \mathcal{I})$ is a
matroid.
\end{theorem}

\begin{proof}
We must prove that $(S, \mathcal{I})$ satisfies (I.1), (I.2) and
(I.3) of Definition~\ref{matro}. Properties (I.1) and (I.2) are
clearly satisfied.

Let us now show that (I.3) holds. Seeking a contradiction, suppose
that (I.3) does not hold. Consider that $I_1 , I_{2} \in
\mathcal{I}$ with $| I_1 | < | I_{2}|$, where $I_1 = \{ i_{a_{1}},
\ldots , i_{a_{j}}\}$ and $I_{2} = \{ i_{b_{1}}, \ldots ,
i_{b_{k}}\}$, $j <  k$. Then for each $e\in I_{2} - I_1$ it
follows that $I_1 \cup \{e\} \notin \mathcal{I}$. We know that
$\{E_{i_{a_{1}}}, E_{i_{a_{2}}}, \ldots , E_{i_{a_{j}}}\}$ and
$\{E_{i_{b_{1}}}, E_{i_{b_{2}}}, \ldots , E_{i_{b_{k}}}\}$ are
LISF's. Since $I_1 \cup \{e\} \notin \mathcal{I}$ for each $e\in
I_{2} - I_1$, then the sets $\{E_{i_{a_{1}}}, E_{i_{a_{2}}},
\ldots ,$ $E_{i_{a_{j}}}, E_{e}\}$ does not form a LISF for each
$e\in I_{2} - I_1$. Fix $e\in I_{2} - I_1$. Then there exist
vectors $ { \bf v}_{l}\in E_{i_{a_{l}}}$, where $1\leq l \leq j$,
and ${ \bf x}\in E_{e}$ such that ${ \bf v}_1, \ldots , { \bf
v}_{j}, { \bf x}$ are linearly dependents in ${\mathbb V}$. Hence
there exist ${\alpha}_{x}, {\alpha}_{l} \in \mathbb{F}$, $1\leq l
\leq j$, with ${\alpha}_{x} \neq 0$ such that ${\alpha}_1 { \bf
v}_{1} + \ldots + {\alpha}_{j} { \bf v}_{j} + {\alpha}_{x} { \bf
x}=0$, otherwise the unique solution for the last equality would
be ${\alpha}_{x}={\alpha}_{l}=0$ for each $1\leq l \leq j$, which
is a contradiction. This means that $ { \bf x}= \left(-{\alpha}_1
{\alpha}_{x}^{-1}\right) { \bf v}_{1} + \ldots +
\left(-{\alpha}_{j} {\alpha}_{x}^{-1}\right) { \bf v}_{j}$. For
every vector ${\bf w} \in E_{e}$ we have ${\bf w}= \beta { \bf x}=
\left(-{\alpha}_1 \beta {\alpha}_{x}^{-1}\right) { \bf v}_{1} +
\ldots + \left(-{\alpha}_{j}\beta {\alpha}_{x}^{-1}\right) { \bf
v}_{j}$, $\beta \in {\mathbb F}$, because $E_{e}\subset
{\mathbb{W}}_{e}$ and ${\mathbb{W}}_{e}$ is an one-dimensional
subspace of ${\mathbb V}$. Thus the subspace ${\mathbb{U}}_{e}$
spanned by the sets $E_{i_{a_{1}}}, E_{i_{a_{2}}}, \ldots $,
$E_{i_{a_{j}}}, E_{e}$, is contained in the subspace $\mathbb{Y}$
spanned by $E_{i_{a_{1}}}, E_{i_{a_{2}}}, \ldots , E_{i_{a_{j}}}$,
for each $e \in I_{2} - I_1$. Consequently, the subspace
$\mathbb{X}$ spanned by $E_{i_{a_{1}}}, E_{i_{a_{2}}}, \ldots ,
E_{i_{a_{j}}}, E_{i_{b_{1}}}, E_{i_{b_{2}}}, \ldots ,$
$E_{i_{b_{k}}}$, is also contained in $\mathbb{Y}$, so it follows
that $|I_{2}| \leq \dim(\mathbb{W})\leq |I_{1}| < |I_{2}|$, which
is a contradiction. Therefore the ordered pair $(S, \mathcal{I})$
is a matroid.
\end{proof}


In the following corollaries of Theorem~\ref{main}, one can
generate more families of matroids:

\begin{corollary}\label{cor1}
Suppose that $n\geq 1$ and $ l\geq 1$ are integers. Let $E_1,
E_{2}, \ldots , E_{n}$ be subsets of a finite dimensional vector
space ${\mathbb V}$ such that $E_{i} \subset {\mathbb{W}}_{i}$ for
all $i=1, \ldots , n$, where ${\mathbb{W}}_{i}$ are
one-dimensional subspaces of ${\mathbb V}$. Consider the multiset
of labels $S =\{ 1, \ldots , n\}$, and let $\mathcal{I}$ be the
set of subsets $I = \{ i_1, \ldots , i_{j}\}$ of $S$ for which
$\{{\lambda_{i_{1}}}E_{i_{1}}, {\lambda_{i_{2}}}E_{i_{2}}, \ldots
,$ $ {\lambda_{i_{j}}}E_{i_{j}}\}$ form a LISF, where
${\lambda_{i_{r}}} \neq 0$ for all $r=1, \ldots, j$. Then the
ordered pair $(S, \mathcal{I})$ is a matroid.
\end{corollary}
\begin{proof}
It follows from the fact that $\{E_{i_{1}}, E_{i_{2}}, \ldots ,
E_{i_{j}}\}$ is a LISF if and only if
$\{{\lambda_{i_{1}}}E_{i_{1}}, {\lambda_{i_{2}}}E_{i_{2}}, \ldots
, {\lambda_{i_{j}}}E_{i_{j}}\}$ is a LISF, where
${\lambda_{i_{r}}} \neq 0$ for all $r=1, \ldots, j$.
\end{proof}

\begin{corollary}\label{cor2}
Consider that $n\geq 1$ and $ l\geq 1$ are integers and let $E_1,
E_{2}, \ldots , E_{n}$ be subsets of a finite dimensional vector
space ${\mathbb V}$ such that $E_{i} \subset {\mathbb{W}}_{i}$ for
all $i=1, \ldots , n$, where ${\mathbb{W}}_{i}$ are
one-dimensional subspaces of ${\mathbb V}$. Assume that $S =\{ 1,
\ldots , n\}$ is the multiset of labels and $\mathcal{I}$ is the
set of subsets $I = \{ i_1, \ldots , i_{j}\}$ of $S$ such that
$\{T(E_{i_{1}}), T(E_{i_{2}}), \ldots ,$ $ T(E_{i_{j}})\}$ form a
LISF, for any isomorphism $T$ on ${\mathbb V}$. Then the ordered
pair $(S, \mathcal{I})$ is a matroid.
\end{corollary}
\begin{proof}
This is true due to the fact that $\{E_{i_{1}}, E_{i_{2}}, \ldots
, E_{i_{j}}\}$ form a LISF if and only if $\{T(E_{i_{1}}),
T(E_{i_{2}}), \ldots , T(E_{i_{j}})\}$ form a LISF.
\end{proof}

\begin{corollary}\label{cor3}
Consider that $n\geq 1$ and $ l\geq 1$ are integers. Let $E_1,
E_{2}, \ldots , E_{n}$ be subsets of a finite dimensional vector
space ${\mathbb V}$ such that $E_{i} \subset {\mathbb{W}}_{i}$ for
all $i=1, \ldots , n$, where ${\mathbb{W}}_{i}$ are
one-dimensional subspaces of ${\mathbb V}$. Consider the multiset
of labels $S =\{ 1, \ldots , n\}$ and let $\mathcal{I}$ be the set
of subsets $I = \{ i_1, \ldots , i_{j}\}$ of $S$ such that
$\{E_{i_{1}}\cup (-E_{i_{1}}), E_{i_{2}}\cup (-E_{i_{2}}), \ldots
, E_{i_{j}}\cup (-E_{i_{j}})\}$ form a LISF. Then the ordered pair
$(S, \mathcal{I})$ is a matroid.
\end{corollary}
\begin{proof}
Follows from the fact that $\{E_{i_{1}}, E_{i_{2}}, \ldots ,
E_{i_{j}}\}$ is a LISF if and only if $\{E_{i_{1}}\cup
(-E_{i_{1}}), E_{i_{2}}\cup (-E_{i_{2}}), \ldots , E_{i_{j}}\cup
(-E_{i_{j}})\}$ is a LISF.
\end{proof}

In the following examples we presents two LISF's that does not
form a matroid:

\begin{example}\label{exa1}
Consider the (real) vector space ${\mathbb R}^{2}$ and the
following subsets $E_1 , E_{2},$ $E_3$ of ${\mathbb R}^{2}$ given
by: $E_1 =\{ (x, y) \in {\mathbb R}^{2} |
{(x-1)}^{2}+{(y-1)}^{2}\leq 1 \}$; $E_{2} =\{ (x, y) \in {\mathbb
R}^{2} | {(x-1)}^{2}+ y^{2}\leq 1 \} \ \{(0, 0)\}$; $E_{3} =\{ (x,
y) \in {\mathbb R}^{2} | {(x-1)}^{2}+ {(y+1)}^{2}\leq 1/9 \}$.
Assume that $S=\{ 1, 2, 3\}$ and consider that $I = \{ i_1, \ldots
, i_{j}\} \in \mathcal{I}$ if and only if the sets $\{E_{i_{1}},
E_{i_{2}}, \ldots , E_{i_{j}}\}$ form a LISF. From construction
and applying the same notation as in Theorem~\ref{main}, one has
$\mathcal{I}=\{\emptyset, \{1\}, \{2,\}, \{3\}, \{1, 3\} \}$. But
since $|\{2\}|<|\{1, 3\}|$, one concludes from (I.3) that $\{1,
2\} \in \mathcal{I}$ or $\{2, 3\} \in \mathcal{I}$, a
contradiction.
\end{example}

\begin{example}\label{exa2}
Consider now the (real) vector space ${\mathbb R}^{3}$ and the
following subsets $E_1 , E_{2}, E_3$ of ${\mathbb R}^{3}$ given
by: $E_1 =\{ (x, y, z) \in {\mathbb R}^{3} | (x, y, z)= (0, 0, 0)+
a(1, 0, 0)+ b(0, 1, 0); a, b \in {\mathbb R} \}$ $- \{ (0, y, 0)|
y \in {\mathbb R} \}$; $E_{2} =\{ (x, y, z) \in {\mathbb R}^{3} |
(x, y, z)= (0, 0, 0)+ c(0, 0, 1)+ d(1, 1, 0); c, d \in {\mathbb R}
\}$ $- \{(0, 0, 0\}$; $E_{3} =\{ (x, y, z) \in {\mathbb R}^{3} |
(x, y, z)= (0, 0, 0)+ e(0, 1, 0)+ f(0, 0, 1); e, f \in {\mathbb R}
\}$  $- \{ (0, y, 0)| y \in {\mathbb R} \}$. As in the previous
example, if $S=\{ 1, 2, 3\}$, from construction we have
$\mathcal{I}=\{\emptyset, \{1\}, \{2,\}, \{3\}, \{1, 3\} \}$.
However, since $|\{2\}|<|\{1, 3\}|$, it follows from (I.3) that
$\{1, 2\} \in \mathcal{I}$ or $\{2, 3\} \in \mathcal{I}$, a
contradiction.
\end{example}

However, under suitable hypothesis one can get the following:

\begin{theorem}\label{main-multdim}
Let ${\mathbb F}$ be a field of characteristic zero. Assume that
${\mathbb V}={\mathbb{W}}_{1}\oplus\ldots \oplus {\mathbb{W}}_{k}$
is a vector space over ${\mathbb F}$ that is the direct sum of
$n$-dimensional subspaces ${\mathbb{W}}_{i}$, $i=1, \ldots , k$.
Let $E_{1}, \ldots , E_{m}$ be subsets of ${\mathbb V}$ such that
for each $i=1, \ldots , m$, $E_{i}$ contains (with exception of
the zero vector) an $n_{E_{i}}$-dimensional subspace of ${\mathbb
V}$ , where $\lceil n/2 \rceil + 1 \leq n_{E_{i}} \leq n$, and
$E_{i}\subset {\mathbb{W}}_{i^{*}}$ for some $1\leq i^{*}\leq k$.
If $S=\{ 1, \ldots, m\}$ is the multiset of labels and
$\mathcal{I}$ is the set of subsets $I = \{ i_1, \ldots , i_{r}\}$
of $S$ such that $\{E_{i_{1}}, E_{i_{2}}, \ldots , E_{i_{r}}\}$
form a LISF, then the ordered pair $(S, \mathcal{I})$ is a
matroid.
\end{theorem}
\begin{proof}
Obviously (I.1) and (I.2) are satisfied. We will prove (I.3). For,
assume that $I_1 , I_{2} \in \mathcal{I}$ with $|I_1| < |I_{2}|$
and $I_1 = \{ i_1, \ldots, i_{s} \}$ and $I_2 = \{ j_1, \ldots,
j_{t} \}$, where $s < t$. Thus the sets $\{E_{i_{1}}, E_{i_{2}},
\ldots , E_{i_{s}}\}$ form a LISF and, from hypothesis, each of
these sets is contained in distinct ${\mathbb{W}}_{i}$'s. These
facts also hold for the sets $\{E_{j_{1}}, E_{j_{2}}, \ldots ,
E_{j_{t}}\}$ corresponding to $I_2$. Suppose without loss of
generality (w.l.g.) that $E_{i_{1}}\subset {\mathbb{W}}_{1},$
$\ldots,$ $E_{i_{s}} \subset {\mathbb{W}}_{s}$. Since $|I_1| <
|I_{2}|$ then there exists an ${\mathbb{W}}_{s^{*}}\neq
{\mathbb{W}}_1, {\mathbb{W}}_{2}, \ldots, {\mathbb{W}}_{s}$ such
that $E_{j_{s+1}} \subset {\mathbb{W}}_{s^{*}}$. This is possible
due to the fact that each of the sets $E_{j_{1}}, E_{j_{2}},
\ldots , E_{j_{t}}$ is contained in distinct ${\mathbb{W}}_{i}$'s
and $s<t$. Consider the sets $E_{i_{1}}, E_{i_{2}}, \ldots ,
E_{i_{s}}, E_{j_{s+1}}$. For every choice of vectors ${{ \bf
v}}_{i_{1}}\in E_{i_{1}},$ \ $\ldots $ \ $,{{ \bf v}}_{i_{s}}\in
E_{i_{s}}$ and ${{ \bf v}}_{j_{s+1}}\in E_{j_{s+1}}$, we claim
that the vectors ${{ \bf v}}_{i_{1}}, \ldots, {{ \bf v}}_{i_{s}},
{{ \bf v}}_{j_{s+1}}$ are linearly independents. In fact, seeking
a contradiction we assume that the vectors are linearly
dependents. W.l.g., suppose that $a_{i_{1}} {{ \bf v}}_{i_{1}}+
\ldots + a_{i_{s}} {{ \bf v}}_{i_{s}}+ b_{j_{s+1}} {{ \bf
v}}_{j_{s+1}}=0$, where $a_{i_{1}}, \ldots, a_{i_{s}}, b_{j_{s+1}}
\in {\mathbb F}$, and $b_{j_{s+1}}\neq 0$; then $\left(a_{i_{1}}
b_{j_{s+1}}^{-1}\right) {{ \bf v}}_{i_{1}}+ \ldots +
\left(a_{i_{s}} b_{j_{s+1}}^{-1}\right) {{ \bf v}}_{i_{s}}+ {{ \bf
v}}_{j_{s+1}}=0$. Since $\left(a_{i_{1}} b_{j_{s+1}}^{-1}\right)
{{ \bf v}}_{i_{1}}\in {\mathbb{W}}_{1}, \ldots , \left(a_{i_{s}}
b_{j_{s+1}}^{-1}\right) {{ \bf v}}_{i_{s}}\in {\mathbb{W}}_{s}$
and ${{ \bf v}}_{j_{s+1}}\in {\mathbb{W}}_{s^{*}}$, then one has
${{ \bf v}}_{j_{s+1}}=0$, which is a contradiction. The cases for
which $a_{i_{1}} {{ \bf v}}_{i_{1}}+ \ldots + a_{i_{s}} {{ \bf
v}}_{i_{s}}+ b_{j_{s+1}} {{ \bf v}}_{j_{s+1}}=0$, where
$a_{i_{1}}, a_{i_{2}}, \ldots, a_{i_{s}}, b_{j_{s+1}} \in {\mathbb
F}$, and $a_{i_{l}}\neq 0$ for some $l=1, \ldots, s$ are
analogous. Thus the sets $\{E_{i_{1}}, E_{i_{2}}, \ldots ,
E_{i_{s}}, E_{j_{s+1}} \}$ form a LISF, so $I_1 \cup \{ j_{s+1}\}
\in \mathcal{I} $, where $j_{s+1}\in I_{2}-I_1$. Therefore, the
ordered pair $(S, \mathcal{I})$ is a matroid and the proof is
complete.
\end{proof}

\section{Summary}

We have constructed new families of matroids derived from linearly
independent set families. The presented construction generalizes
in a natural way the class of vectorial matroids over a field.

\section*{Acknowledgment}
This research was partially supported by the Brazilian Agencies
CAPES and CNPq.


\begin{thebibliography}{5}


\bibitem{Brualdi:1969}
R. A. Brualdi, Comments on bases in dependence structures, Bull.
Australian Math. Society, {\bf 1} 161--167, 1969.

\bibitem{Brylawski:1973}
T. H. Brylawski, Some Properties of Basic Families of Subsets,
Discrete Math., {\bf 6} 333--341, 1973.


\bibitem{Eicker:2004}
F. Eicker and G. Ewald, Linearly independent set families, Linear
Algebra and its Applications,  {\bf 388} 173--191, 2004.

\bibitem{Greene:1976}
C. Greene, Weight enumeration and the geometry of linear codes,
Stud. Appl. Math., {\bf 55} 119--128, 1976.

\bibitem{Oxley:1992}
J. G. Oxley, \textit{Matroid Theory}, Oxford University Press, New
York, 1992.

\bibitem{Welsh:1976}
D. J. A. Welsh, \textit{Matroid Theory}, Academic Press Inc.,
London L.T.D., 1976.

\bibitem{Whitney:1935}
H. Whitney, On the abstract properties of linear dependence, Amer.
J. Math., {\bf 57} 509--533, 1935.

\end{thebibliography}
\end{document}